\documentclass[11pt]{amsart}
\usepackage[margin=1in]{geometry}
\usepackage{version, amsmath, amssymb, mathtools}
\usepackage{color}
\usepackage{amsthm, amsopn, amsfonts}
\usepackage[colorlinks=true]{hyperref}
\usepackage{labelschanged}
\usepackage{comment}

\usepackage[margin=1in]{geometry}
\usepackage{version, mathtools}
\usepackage{color}
\usepackage{amsopn}
\usepackage{mathrsfs}
\usepackage{slashed}
\usepackage{physics}
\usepackage{breqn}
\usepackage[nocompress]{cite}
\usepackage[all,cmtip]{xy}
\usepackage{esint}
\usepackage{enumitem}

\newcommand{\pa}{{\partial}}

\renewcommand{\Im}{\text{Im}}

\newtheorem{theorem}{Theorem}[section]

\newtheorem{lemma}[theorem]{Lemma}
\newtheorem{corr}[theorem]{Corollary}

\newtheorem{proposition}[theorem]{Proposition}
\newtheorem{definition}[theorem]{Definition}
\newtheorem{remark}[theorem]{Remark}

\newtheorem{assumption}{Assumption}

\newcommand{\supp}{\text{supp }}
\newcommand{\R}{{\mathbb R}}

\newcommand{\la}{\langle}
\newcommand{\ra}{\rangle}

\renewcommand{\1}{{\mathbf 1}}
\newcommand{\e}{\epsilon}

\newcommand{\pd}{\partial}

\newcommand{\qn}[1]{{\left\vert\kern-0.25ex\left\vert\kern-0.25ex\left\vert #1 
    \right\vert\kern-0.25ex\right\vert\kern-0.25ex\right\vert}}

\newcommand{\tP}{\tilde{P}}

\newcommand{\A}{\mathcal{A}}
\newcommand{\opw}{\text{Op}^{\mathrm{w}}}

\begin{document}
\title{Local Energy Decay for Non-Stationary Damped Wave Operators}
\author{Nicholas Arsenault}
\address{Department of Mathematics, University of Kentucky, Lexington,
  KY 40506}
\email{nick.arsenault@uky.edu}
\author{Mihai Tohaneanu}
\address{Department of Mathematics, University of Kentucky, Lexington,
  KY 40506}
\email{mihai.tohaneanu@uky.edu}
\begin{abstract}
   The paper establishes integrated local energy decay (ILED) estimates for the damped wave equation on certain non-stationary spacetimes. The main technical result is a high frequency estimate that holds in great generality, provided that null geodesics trapped in a compact region are sufficiently damped. 
\end{abstract}
\maketitle
\emph{\emph{}}

\section{Introduction}

The damped wave equation is used in physics as a model to describe various phenomena involving wave propagation in the presence of friction. A relevant example is the equation 
\begin{equation}\label{dampcpt}
\partial_t^2 u - \Delta_g u + a(x)\partial_t u = 0.    
\end{equation}
Here $(M, g)$ is a three-dimensional Riemannian manifold,  $\Delta_g$ is the associated Laplace-Beltrami operator, and $a\in C^{\infty}(M)$ is the damping function satisfying $a(x)\geq 0$.

Let
\[
E[u](T) = \frac12 \int_{M} |\pa_t u(T, x)|^2 + g^{ij}\pa_i u\pa_j u dV_g \approx \int_{M} |\pa_{t, x} u(T, x)|^2 dV_g
\]
denote the associated energy. Multiplying \eqref{dampMink} by $\partial_t u$ and integrating by parts immediately yields that
\[
E[u](T) = E[u](0) - \int_0^T \int_{\R^3} a(x)|\partial_t u(t,x)|^2 dx dt \leq E[u](0)
\]
which establishes uniform boundedness of the energy. 

A relevant question to ask is what conditions one needs to impose on $a(x)$ so that the energy decays. In the case of compact manifolds $M$, one can prove exponential decay of the energy, under the geometric control condition (GCC), namely there is a fixed $L>0$ so that every null bicharacteristic intersects the set $\{a(x)>0\}$ in time less than $L$; see, for example, \cite{RT75}, \cite{BLR92}, \cite{kleinhenz2023energydecaytimedependent}. More precisely, one has
\begin{equation}\label{expcpt}
E[u](T) \leq Ce^{-ct} E[u](0)
\end{equation}
for some constants $C$, $c$. When the assumption fails, assuming extra regularity on the initial data, one can generically prove that the energy decays at most logarithmically, though the rate could be better depending on various geometries; see for example \cite{burq1998decroissance}, \cite{Leb96}, \cite{Chr07}, \cite{LiuRao05}, \cite{AnLe14}, \cite{kleinhenz2025sharpenergydecayrates}.      

Things are considerably different if the manifold $M$ is not compact. Even in the most basic case
\begin{equation}\label{dampMink}
\partial_t^2 u - \Delta u + a(x)\partial_t u = 0,    
\end{equation}
\eqref{expcpt} does not hold even if GCC does (though it does hold for the Klein-Gordon equation \cite{BJ16}). 
More refined arguments yield pointwise decay and Strichartz estimates of the solution to \eqref{dampMink} \cite{Mat76}, \cite{IIOW19}. 

In the present paper we will be interested in establishing integrated local energy estimates (ILED) for damped wave operators \eqref{op}. Such estimates originated in the work of Morawetz \cite{Mor68} and have played a crucial role in understanding the dispersive properties of linear solutions to undamped waves, as well as global well-posedness and asymptotic behavior of solutions to nonlinear equations. There has been an enormous amount of work done in this direction for undamped operators, see the references in Section 2. On the other hand, local energy estimates for the damped wave operators and variable coefficients have also been recently studied, but only in the case of time-independent coefficients, see \cite{Bouclet_2014}, \cite{kofroth2025integratedlocalenergydecay}, \cite{KleinhenzMcNulty2025energydecaytimedependent}.

Our main result Theorem~\ref{main} establishes ILED for damped wave operators with coefficients that may be allowed to depend on time. There are two parts to it. The first part, which is the main contribution of the paper, establishes a two point ILED estimate: with minimal assumptions on the coefficients (in particular no smallness assumptions), and assuming that the GCC holds (see Assumption~\ref{ass:GCC}), we can control the local energy on a time interval $[0,T]$ by the energy at initial and final time and lower order spacetime terms. The second part is an unconditional ILED result, under the additional assumption that the coefficients are almost stationary (see Assumption~\ref{asmp:sv}). It would be interesting to understand to what degree Assumption~\ref{asmp:sv} can be weakened, and if ILED holds without any smallness assumptions in this context.

\subsection{Acknowledgments} The authors are grateful to Jason Metcalfe, Perry Kleinhenz,  Collin Kofroth, and Daniel Tataru, for conversations about their results.

\section{Preliminaries and Setup}
\label{sec:notation}
\subsection{The Damped Wave Operator}

Let $(\mathbb{R}^4, g)$ denote a Lorentzian manifold, where $g$ is a metric with signature $(-,+,+,+)$. Throughout this work, we will study operators of the form
\begin{equation}\label{op}
    P = \Box_g + i a D_t,
\end{equation}
where $\Box_g = D_\alpha g^{\alpha\beta} D_\beta$ and $D_\alpha = \frac{1}{i} \partial_\alpha$. We adopt the standard convention that Greek indices $\alpha, \beta,\ldots$ range from $0$ to $3$, while Latin indices $i, j, k, \ldots$ range from $1$ to $3$. The Einstein summation convention is used throughout.

The function $a \in C_c^\infty(\mathbb{R}^3)$ is non-negative and positive on an open set; we call $a$ the \emph{damping function}. The presence of the damping term $ia D_t$ models energy dissipation in the system and is crucial for obtaining decay estimates.

\subsection{Local Energy Norms}

To study the behavior of solutions to the damped wave equation, we introduce several weighted mixed-norm spaces. We work with spaces $L^p L^q$, where the $L^p$ norm is taken in the time variable and $L^q$ in the spatial variables. For $R \geq 1$, let $A_R = \{x \in \mathbb{R}^3 : R \leq \langle x \rangle \leq 2R\}$ denote the dyadic annulus at scale $R$, where $\langle x \rangle = (1 + |x|^2)^{1/2}$ is the Japanese bracket. We also define 
\[
A_{\leq R} = \{x \in \mathbb{R}^3 : \langle x \rangle \leq 2R\}
\]
\[
A_{\geq R} = \{x \in \mathbb{R}^3 : R \leq \langle x \rangle \}
\]

\begin{definition}[Local energy norms]
We define the local energy norm by
\begin{align}
    \|u\|_{LE} &= \sup_{j \in \mathbb{N}} \|\langle x \rangle^{-\frac{1}{2}} u\|_{L^2 L^2(\mathbb{R}_+ \times A_{2^j})}, \\
    \|u\|_{LE^1} &= \|u\|_{LE} + \|\langle x \rangle^{-1} \partial u\|_{LE},
\end{align}
where $\partial = (\partial_t, \nabla)$ denotes the full space-time gradient.
\end{definition}

For the inhomogeneous term in our wave equation, we require the dual-type norm
\begin{equation}
    \|F\|_{LE^*} = \sum_{j=0}^\infty \|\langle x \rangle^{\frac{1}{2}} F\|_{L^2 L^2(\mathbb{R}_+ \times A_{2^j})},
\end{equation}
and the mixed norm
\begin{equation}
    \|F\|_{L^1 L^2 + LE^*} = \inf_{F = f_1 + f_2} \left(\|f_1\|_{L^1 L^2} + \|f_2\|_{LE^*}\right).
\end{equation}

We also define by $LE[t_0, t_1]$ the $LE$ norm restricted to the time interval $[t_0, t_1]$ and similarly for all the other norms.


We will be interested in establishing an integrated local energy estimate of the form
\begin{equation}\label{localenergyflat}
\|\partial u\|_{L^{\infty}_t L^2_x} + \| u\|_{LE^1}
 \lesssim \|\partial u(0)\|_{L^2} + \|Pu \|_{LE^*+L^1_t L^2_x}
\end{equation}
and a similar estimate involving the $LE^1[t_0, t_1]$ and $LE^*[t_0, t_1]$ norms.

The first estimate of this kind was obtained by Morawetz for the Klein-Gordon equation \cite{Mor68}. In the self-adjoint case $P=\Box_g$, there are many such results for asymptotically flat,
small perturbations of the Minkowski space-time (\cite{KSS02, KPV95, SmSo00, St05, Strauss75,
Al06, MS06, MT09}). Even for large perturbations, in the absence of trapping, \eqref{localenergyflat} still sometimes holds, see for instance \cite{BH10}, \cite{MST20}. In the presence of trapping, \eqref{localenergyflat} is known to fail, see \cite{Ral69}, \cite{Sb15}. However, a weaker form of \eqref{localenergyflat} may hold if the trapping is normally hyperbolic, if one allows for a loss of derivatives. This is the case for the Schwarzschild \cite{BS03,BS05,BSt06, MMTT10 ,DR09, DR07} and Kerr \cite{TT11, DR08, AB15, DRS16} black hole backgrounds in General Relativity.

\subsection{Asymptotically Flat Metrics}

To quantify how the metric $g$ approaches the Minkowski metric $m = \text{diag}(-1, 1, 1, 1)$ at spatial infinity, we introduce the asymptotically flat norm.

\begin{definition}[AF norm]
For a metric $g$ on $[0,T] \times \mathbb{R}^3$, and $0\leq\alpha,\beta\leq 3$, we define
\begin{equation}\label{eq:AF}
    \|g\|_{AF} = \sup_{0\leq\alpha,\beta\leq 3}\sum_{|\gamma| \leq 2} \|\langle x \rangle^{|\gamma|} \partial^\gamma g^{\alpha\beta}\|_{\ell_j^1 L^\infty([0,T] \times A_{2^j})}.
\end{equation}
More generally, for a cutoff region, we write $\|g\|_{AF_{>R_0}}$ to denote the norm restricted to $|x| > R_0$, and $\|g\|_{AF(A_{2^j})}$ for the norm restricted to the dyadic shell $A_{2^j}$.
\end{definition}

\begin{definition}[Asymptotically flat operator]
\label{def:asymptotically_flat}
We say that $P$ is \emph{asymptotically flat} if $\|g - m\|_{AF} < \infty$.
\end{definition}

In particular, we have the important bound
\begin{equation}\label{AFcons}
|\partial g^{\alpha\beta}(t, x)| \lesssim \la x\ra^{-1}.     
\end{equation}
\subsubsection{Quantitative Parameters}

For technical reasons, we need precise quantitative control on how the metric approaches the Minkowski metric. We introduce the following parameters that will be used throughout.

First, we fix a small constant $\mathbf{c} > 0$ and choose a radius $R_0 > 0$ sufficiently large such that
\begin{equation}\label{eq:R_0}
    \|g - m\|_{AF_{>R_0}} \leq \mathbf{c} \ll 1.
\end{equation}
This ensures that outside a ball of radius $R_0$, the metric $g$ is uniformly close to the Minkowski metric $m$.

Second, for each dyadic shell with $j > \log_2 R_0$, we define coefficients $\{c_j\}$ satisfying
\begin{equation}
    \|g - m\|_{AF(A_{2^j})} \lesssim c_j, \quad \text{and} \quad \sum_j c_j \lesssim \mathbf{c}.
\end{equation}
We may additionally assume that the sequence $\{c_j\}$ satisfies a slow variation condition:
\begin{equation}
    \frac{c_j}{c_k} \leq 2^{\delta |k-j|}, \quad \text{for some} \quad \delta \ll 1.
\end{equation}
This slow variation condition ensures that the rate of approach to the Minkowski metric does not oscillate too rapidly between adjacent dyadic scales, which is crucial for the propagation of estimates across different spatial regions.

These quantitative controls provide the foundation for treating $P$ as a small perturbation of the flat wave operator in the far field, enabling the use of perturbative techniques in our analysis.

We will take $R_0$ large enough so that the support of $a$ is contained in $\{|x|\leq R_0\}$.

Finally, we let $\tilde P=\Box_{\tilde g}$ be a small AF flat perturbation of $\Box_g$, in the sense that $\tilde P=\Box_g$ when $|x|\geq R_0$, and so that \eqref{localenergyflat} holds for $\tilde P$.

\subsection*{Timelike and Spacelike Foliation Assumptions}

Two crucial geometric assumptions in our analysis are the following:
\begin{enumerate}
    \item \emph{Uniformly timelike time-translation.} The coordinate vector field
    $\partial_t=\partial/\partial x^0$ is uniformly timelike, i.e. there exists $c_0>0$ such that
    \begin{equation}\label{eq:uniform_timelike_vector}
        g(\partial_t,\partial_t)=g_{00}(t,x)\le -c_0<0
        \qquad \text{for all } (t,x)\in \R_+\times\R^3 .
    \end{equation}

    \item \emph{Uniformly spacelike time-slices.} The hypersurfaces $\{t=\mathrm{const}\}$ are
    uniformly spacelike. Equivalently, the covector field $dt$ is uniformly timelike with respect
    to the dual metric $g^{-1}$, i.e. there exists $c_1>0$ such that
    \begin{equation}\label{eq:uniform_timelike_covector}
        g^{-1}(dt,dt)=g^{00}(t,x)\le -c_1<0
        \qquad \text{for all } (t,x)\in \R_+\times\R^3 .
    \end{equation}
\end{enumerate}
Under these assumptions, the spatial part of the metric induces a uniformly elliptic operator. Specifically, the operator $D_i g^{ij} D_j$ is uniformly elliptic in the sense that
\begin{equation}
    \label{eq:uniform_ellipticity}
    g^{ij}\xi_i \xi_j \approx |\xi|^2, \quad \text{for all } \xi \neq 0,
\end{equation}
where the implicit constants are uniform in space and time. This ellipticity is essential for energy estimates and provides control on spatial derivatives.

\subsection{Cutoff Functions}

To localize our analysis to different spatial regions, we introduce a family of smooth cutoff functions that will be used throughout the subsequent proofs.

Let $\chi \in C_c^\infty(\mathbb{R})$ be a non-increasing function satisfying
\begin{equation}
    \chi(x) = \begin{cases}
        1, & |x| \leq 1, \\
        0, & |x| > 2,
    \end{cases}
\end{equation}
with $0 \leq \chi \leq 1$ everywhere. For any radius $R > 0$, we define the scaled cutoff functions
\begin{align}
    \chi_{<R}(|x|) &:= \chi\left(\frac{|x|}{R}\right), \quad \text{(supported in $|x| \lesssim R$)} \\
    \chi_{>R}(|x|) &:= 1 - \chi_{<R}(|x|), \quad \text{(supported in $|x| \gtrsim R$)}.
\end{align}
These functions allow us to separate contributions from the interior region $|x| < R$ and the exterior region $|x| > R$.

Finally, we let $\chi_R \in C_c^\infty(\mathbb{R})$ denote a cutoff function satisfying $0 \leq \chi_R \leq 1$ with $\operatorname{supp} \chi_R \subset \{x : |x| \approx R\}$. This localizes to a neighborhood of the sphere of radius $R$ and will be particularly useful when we perform dyadic decompositions.

The derivatives of these cutoff functions satisfy the natural bounds
\begin{equation}
    |\partial^\alpha \chi_{<R}| + |\partial^\alpha \chi_{>R}| \lesssim_\alpha R^{-|\alpha|}, \quad |\partial^\alpha \chi_R| \lesssim_\alpha R^{-|\alpha|},
\end{equation}
which must be tracked carefully in commutator estimates.

\subsection{The Geometric Control Condition}

Our analysis rests crucially on the \emph{Geometric Control Condition (GCC)}, which we state precisely as:

\begin{assumption}[Geometric Control Condition]\label{ass:GCC}
Fix \(R_0\) from \eqref{eq:R_0}. Then there exists \(L>0\) and \(\delta>0\) such that for $w \in \Sigma$ and all $s_0 \in \R$:
if $x_{[0,L]}(w)\subset \{|x|\le R_0\}$, then
\[
\frac1L\int_{0}^{L} a\bigl(x_s(w)\bigr)\,ds \ge \delta.
\]
\end{assumption}

This is similar to the GCC of \cite{KleinhenzMcNulty2025energydecaytimedependent}.

This condition requires that
\begin{enumerate}
\item For any null-bicharacteristic and any time interval $I\subset \R$ of length $L$, if the trajectory stays in the interior region throughout that interval,
\[
x_I(w)\subset \{|x|\le R_0\},
\]
then it must spend accumulate enough damping in the interval.
\item Whenever a null-bicharacteristic remains in $\{|x|\le R_0\}$, the time gaps between visits to damping are uniformly bounded by $L$.
\end{enumerate}

The crucial piece is the \emph{uniformity} of the constants $L$ and $\delta$ over all null-bicharacteristics, in both forward and backward flow directions, and over all time windows. This uniform control prevents non-trapped null-bicharacteristics from staying in the interior region for arbitrarily long times without interacting with damping, and then leaving. This allows us to construct global escape functions with time-independent symbol bounds, leading to estimates with constants independent of $t$.

For the second part of our theorem, we will also need the following, which is the definition of $\epsilon-$ slowly varying from \cite{MST20}. 
\begin{assumption}[Slow variation in time]\label{asmp:sv}
There exists $0<\epsilon \ll 1$ such that
\begin{equation}\label{eq:sv}
  \|\partial_t g\|_{AF} < \epsilon.
\end{equation}
\end{assumption}

\subsection{The Main Result}

We are now ready to state the main result of the paper.

\begin{theorem}
\label{main}
\begin{enumerate}[label=(\roman*)]
\item
Let $P$ be an asymptotically flat damped wave operator satisfying \eqref{eq:uniform_timelike_vector}- \eqref{eq:uniform_ellipticity} and Assumption \ref{ass:GCC}. Then the following two point estimate holds:
\[
\|u\|_{LE^1[0, T]} + \|\partial u\|_{L^\infty[0, T]L^2} \lesssim E[u](0) + E[u](T) + \|u\|_{L^2[0, T]L^2(A_{\leq R_0})} + \|Pu\|_{LE^*[0, T] + L^1 L^2[0, T]}.
\]

\item 

If in addition we assume that Assumption~\ref{asmp:sv} holds, ILED and uniform energy bounds hold. More precisely 
\begin{equation}\label{ILED}
\|u\|_{LE^1[0, T]} + \|\partial u\|_{L^\infty[0, T]L^2} \lesssim E[u](0) +  \|Pu\|_{LE^*[0, T] + L^1 L^2[0, T]}.
\end{equation}

\end{enumerate}
All the implicit constants are independent of $T$. 
\end{theorem}

Our theorem extends the result of Kofroth \cite{kof23} to the nonstationary case. In the stationary case, an elementary compactness argument ensures that Assumption~\ref{ass:GCC} is equivalent to all trapped null geodesics intersecting the damped set, which was the assumption used in \cite{kof23}. In the nonstationary case, the latter assumption is not quite enough, as it is possible for geodesics to linger in a compact region for an arbitrary length of time. Another issue we need to contend with is the lack of a conserved coercive energy, which is why we obtain only a two point estimate in part (i) of Theorem~\ref{main}.

\section{High Frequency Analysis}
\label{chap:high_freq}
The high-frequency regime constitutes the central technical challenge in establishing local energy decay for damped wave operators on non-stationary space-times. This difficulty stems from the phenomenon of trapping. In contrast to the low- and medium-frequency regimes, where energy either disperses to infinity or can be controlled by Carleman-type estimates with large parameters, high-frequency waves concentrate along null-bicharacteristics and can exhibit persistent localized energy when these trajectories avoid damped regions.

The fundamental obstacle is that trapped null-bicharacteristics may fail to interact sufficiently
with the damping. In the stationary setting studied in \cite{kof23}, the temporal and spatial variables can be decoupled, allowing one to work on $T^*\R^3$. Additionally, a compactness argument provides a uniform escape time for non-trapped null-bicharacteristics. However, in our non-stationary setting, the time-dependence of the phase space structure necessitates working on $T^*\mathbb{R}^4$ rather than $T^*\mathbb{R}^3$, and loss of compactness in the temporal direction requires control for non-trapped null-bicharacteristics that linger in a compact set.

\subsection{Rescaling}

It will be convenient to have a large constant $\gamma > 0$ in front of the damping term. Following \cite{kof23} and \cite{KleinhenzMcNulty2025energydecaytimedependent}, we rescale the problem to
\begin{equation}\label{resc}
\tilde P \tilde u = \tilde F    
\end{equation}
where 
\[
\tilde P = D_\alpha \tilde g^{\alpha\beta}D_\beta + i\gamma \tilde a D_t
\]
and
\[\begin{split}
\tilde u(t, x) = \gamma^{-2} u(\gamma t, \gamma x) \\
\tilde F(t, x) = F(\gamma t, \gamma x) \\
\tilde g^{\alpha\beta}(t, x) = g^{\alpha\beta}(\gamma t, \gamma x) \\
\tilde a(x) = a(\gamma x) \\
\tilde R_0 = \gamma R_0
\end{split}\]

It is important to notice that Assumption \ref{ass:GCC} still holds under rescaling, with the same constants $L$ and $\delta$. Moreover, we also have, due to \eqref{AFcons} that
\begin{equation}\label{resc}
|\pa_x \tilde g^{\alpha\beta}(t, x)| = |\gamma (\pa_x g^{\alpha\beta})(\gamma t, \gamma x)| \lesssim \|\pa_x g^{\alpha\beta}(t, x)\|_{L^\infty}
\end{equation}
with a constant that is independent of $\gamma$. 
We will from now on consider \eqref{resc}, and drop the $\tilde{}$ for convenience. 

\subsection{The Positive Commutator Method}

Our proof of the main high-frequency estimate follows the positive commutator method. The key idea is to construct a pseudodifferential operator $Q = q^w(t,x,\tau,\xi)$ satisfying
\[
\mathrm{Re}\langle i[P,Q]u, u\rangle + 2\mathrm{Re}\langle \gamma a Q u, u\rangle + \|\langle x \rangle^{-2} u\|_{LE}^2 \gtrsim \|u\|_{LE^1}^2
\]
where $\gamma > 0$ is large enough, and $a(x)$ is the damping. At the symbol level, this translates to the following:

\begin{lemma}
\label{mainlemma}
 There exist symbols $q \in S_{hom}^1(T^*\mathbb{R}^4)$ and $m \in S_{hom}^0(T^*\mathbb{R}^4)$ so that
\[
H_p q + 2\gamma \tau a(x) q + p m \gtrsim  \langle x \rangle^{-2}(\tau^2 + |\xi|^2),
\]
where the implicit constant is independent of time $t$.
\end{lemma}
The proof proceeds through several steps:
\begin{description}
  \item[Section \ref{sec:hamilton_flow}] We establish fundamental properties of the bicharacteristic flow $\varphi_s : T^*\mathbb{R}^4 \setminus \{0\} \to T^*\mathbb{R}^4 \setminus \{0\}$ associated to the principal symbol $p(t,x,\tau,\xi) = \tau^2 - 2\tau g^{0j}\xi_j - g^{ij}\xi_i\xi_j$. The time-dependence of $g$ requires careful attention to uniformity in $t$

\item[Section \ref{sec:semi_trapped}]  This constitutes the core innovation. For $L$-trapped null-bicharacteristics a symbol is constructed that accumulates damping seen either forward or backward in flow time. It is crucial that the constants in this estimate be uniform, and it is here that we use the GCC.

\item[Section \ref{sec:non_trapped}] For null-bicharacteristics escaping to spatial infinity, we follow the method in \cite{kof23}, but adopt it to the non-stationary setting. In particular, in this region null-bicharacteristics exit $B(0,2R_0)$ within flow time $L$. We construct the symbol that measures how long the null-bicharacteristic will remain inside this ball. Finally, to handle error terms created we add an exterior correction coming from a modified Morawetz multiplier.
\item[Section \ref{pf:mainlemma}] After constructing the symbols on the $L$-trapped and interior non-trapped sets, we combine these constructions on each light-cone, providing the desired positivity on the characteristic set. To obtain positivity on the elliptic set we construct a Langrangian correction. Our proof deviates from that of \cite{MST20, kof23, KleinhenzMcNulty2025energydecaytimedependent} and does not rely on the symbol $q$ being affine in $\tau$.


\end{description}

\subsection{Results on the Hamilton Flow}\label{sec:hamilton_flow}
From the assumption that $\pd_t$ is uniformly time-like, the signature of the metric, and the
co-factor expansion, we have that $g^{00} \lesssim -1$. Using this we divide the operator by
$g^{00}$, so that, after doing so, $g^{00} = 1$. Such an operation preserves the assumptions placed
on the coefficients of operator. Moving forward, we assume  without loss of generality, that $g^{00}
= 1$ and, since $g^{00}$ was initially negative, we can denote $g^{ij}/g^{00}$ by $-g^{ij}$. This
produces an operator having the symbol
\begin{equation}
    \sigma(P)(t, x, \tau, \xi)=\tau^2-2 \tau g^{0 j}(t,x) \xi_j-g^{i j}(t,x) \xi_i \xi_j .
\end{equation}

Furthermore, denote the principal symbol of symmetric and skew-symmetric parts of $P$ by
\begin{align*}
    p(t,x,\tau, \xi) &= \tau^2 - 2 \tau g^{0 j}(t,x) \xi_j - g^{ij}(t,x) \xi_i \xi_j \\
    s_{skew}(x,\tau,\xi) &= i \gamma \tau a(x). 
\end{align*}
Factoring the principal symbol of the symmetric part gives $$p(t,x,\tau,\xi) = p^+(t,x,\tau,\xi)p^-(t, x, \tau,\xi),$$
where $p^{\pm}(t,x,\tau,\xi) = \tau - b^\pm(t,x,\xi)$. It can be seen that
\begin{equation}
    b^\pm(t,x,\xi) = g^{0j} \xi_j \pm \sqrt{(g^{0j}(t,x)\xi_j)^2 + g^{ij} \xi_i \xi_j}.
\end{equation}
Moreover for fixed $\eta > 0$, 
\begin{align*}
    b^\pm(t,x,\eta \xi) = g^{0j} \eta \xi_j \pm \sqrt{(g^{0j}(t,x) \eta \xi_j)^2 + g^{ij} \eta^2
    \xi_i \xi_j} = \eta b^{\pm}(t,x,\xi).
\end{align*}
Thus, $b^\pm$ is homogeneous of degree 1. 

The symbol $p$ generates a Hamilton flow, $\varphi_s(w) \in \R \cross T^*\R^4$, given by
\begin{equation*}
    \begin{cases}
        \dot{t}_s = \pd_\tau p(\varphi_s(w)), & \dot{\tau}_s = -\pd_t p(\varphi_s(w)) \\
      (\dot{x}_s)_k = \pd_{\xi_k} p(\varphi_s(w)), & \dot{\xi}_s = \pd_{x_k} p(\varphi_s(w))    \end{cases}
\end{equation*}
which for our operator, gives
\begin{equation*}
    \begin{cases}
        \dot{t}_s = 2 \tau_s - 2 g^{0j}(t_s, x_s) [\xi_s]_j, & \dot{\tau}_s = -2\tau_s \pd_{t} g^{0j}(t_s, x_s)[\xi_s]_j - \pd_{t} g^{ij} (t_s, x_s) [\xi_s]_i [\xi_s]_j \\
        (\dot{x}_s)_k = -2 \tau_s g^{0k}(t_s, x_s) - 2 g^{kj}(t_s, x_s) [\xi_s]_j, & (\dot{\xi}_s)_k = -2\tau_s \pd_{x_k} g^{0j}(t_s, x_s)[\xi_s]_j - \pd_{x_k} g^{ij} (t_s, x_s) [\xi_s]_i [\xi_s]_j.
    \end{cases}
\end{equation*}
Additionally, we denote the flow generated by $p^\pm$ by $\varphi_s^\pm(w) = (t_s^\pm(w), x_s^\pm(w), \tau_s^\pm(w), \xi_s^\pm(w))$. This is a solution to the system
        \begin{align}\label{propsys2}
        \left\{\begin{aligned}
            \dot{t}^+_s &= 1 \\
            \dot{\tau}^+_s &=  \pd_t b^+(\varphi_s(w)) \\
            (\dot{x}_s)^+_k &= p^+_{\xi_k}(\varphi_s(w)) \\
            (\dot{\xi}^+_s)_k &= -p_{x_k}^+(\varphi_s(w))  \\
            (t^+_0, \tau^+_0, x^+_0, \xi^+_0) &= w
        \end{aligned}
        \right.
    \end{align}

We now state a useful proposition that allows us to reduce the construction to the cosphere bundle.
\begin{proposition}\label{rof:homogeneity}
    The Hamilton vector field $H_{p^\pm}$ is homogeneous of degree $0$. That is, denoting dilation in the fiber by $M_\lambda(t, x,\tau, \xi) = (t, x,\lambda \tau, \lambda \xi)$, $\lambda > 0$, then
    \begin{equation}
        M_\lambda^*(H_{p^\pm}|_{(t_0, x_0,\tau_0, \xi_0)}) q = H_{p^\pm}|_{(t_0, x_0, \tau_0, \xi_0)} q.
    \end{equation}
\end{proposition}

\begin{proof}
  We note that  $p^\pm \in S^1_{hom}(T^*\R^4)$, while $\pd_{x_j}p^\pm \in S^1_{hom}(T^*\R^4)$ and $\pd_{\xi_j} \in S^0_{hom}(T^*\R^4)$. Let $q \in C^\infty(T^*\R^4)$. By direct computation we have
    \begin{align*}
        M^*_\lambda (H_{p^\pm}|_{(t_0,x_0,\tau_0,\xi_0)}) q &= H_{p^\pm}|_{(t_0, x_0, \lambda \tau_0, \lambda \xi_0)} (q \circ M^{-1}_\lambda) \\
        &= (\nabla_\xi p^\pm)(t_0, x_0, \lambda\tau_0, \lambda\xi_0) \cdot (\nabla_x (q \circ M^{-1}_\lambda))(t_0, x_0, \lambda\tau_0, \lambda\xi_0) \\
        &\quad - (\nabla_x p^\pm)(t_0, x_0, \lambda\tau_0, \lambda\xi_0) \cdot (\nabla_\xi (q \circ M^{-1}_\lambda))(t_0, x_0, \lambda\tau_0, \lambda\xi_0) \\
        &= (\nabla_\xi p^\pm)(t_0, x_0, \tau_0, \xi_0) \cdot (\nabla_x q)(t_0, x_0, \tau_0, \xi_0) \\
        &\quad - \lambda (\nabla_x p^\pm)(t_0, x_0, \tau_0, \xi_0) \cdot \lambda^{-1}(\nabla_\xi q)(t_0, x_0, \tau_0, \xi_0)\\
        &= H_{p^\pm}|_{(t_0,x_0, \xi_0, \tau_0)}.
    \end{align*} \end{proof}

Fix any smooth positive ``radius'' \(r=r(\tau,\xi)\) which is homogeneous of degree \(1\) in \((\tau,\xi)\) (e.g.
\(r(\tau,\xi)=\sqrt{\tau^2+|\xi|^2}\)), and set
\[
S^*\R^4:=\{(t,x,\tau,\xi): r(\tau,\xi)=1\}.
\]

Define the projection to the chosen cosphere section by
\[
\pi(\rho):=M_{1/r(\rho)}(\rho),\qquad \rho\in T^*(\R^{4})\setminus o.
\]
Since \(\pi\circ M_\lambda=\pi\), the differential \(d\pi\) annihilates the vector field
\(R:=\tau\partial_\tau+\xi\cdot\partial_\xi\).
The \emph{descended} (or induced) Hamilton vector field on \(S^* \R^4\) is then defined by
\[
\widehat H_{p^\pm}:=d\pi(H_{p^\pm}),
\]
and can be represented on the section \(S^*\) by subtracting the radial component
\[
\widehat H_{p^\pm}
= H_{p^\pm} - \frac{H_{p^\pm}r}{Rr}\,R
\qquad \text{on }S^*\R^4.
\]
Since \(r\) is degree \(1\), \(Rr=r\), hence on \(S^*\) this simplifies to
\[
\widehat H_{p^\pm}=H_{p^\pm}-(H_{p^\pm}r)\,R.
\]
In particular, \(\widehat H_{p^\pm}\) is tangent to \(S^*\) (indeed \(\widehat H_{p^\pm}r=0\) on \(S^*\)).

Finally, if \(q_1\in C^\infty(S^*)\) and we extend it to \(T^*\R^{4}\setminus 0\) by the \(0\)-homogeneous extension
\(q:=q_1\circ\pi\), then the chain rule yields
\[
H_{p^\pm}q = (\widehat H_{p^\pm}q_1)\circ \pi.
\]
Consequently, any pointwise inequality proved on the cosphere bundle for \(\widehat H_{p^\pm}q_1\) lifts to the corresponding inequality for \(H_{p^\pm}q\) on
\(T^*\R^{4}\setminus o\) after homogeneous extension.

\begin{remark}
    The radial function used in \cite{kof23}, \cite{KleinhenzMcNulty2025energydecaytimedependent} is $b(x,\xi)$ coming from the the half-wave decomposition. Due to $g$ being stationary, $b$ is a constant of motion, $H_p b = 0$. Consequently, $H_p$ descends to $\{b = 1\}$ as $\widehat{H}_p = H_p$.
\end{remark}

Fix $R>0$ and $L>0$. To capture bicharacteristics which may be non-trapped but can remain in
$\{|x|\le R_0\}$ for a long time, we introduce  \emph{$L$-trapped sets}.
Define the forward and backward \emph{exit times} from the ball $\{|x|\le R_0\}$ by
\[
T_{R_0}^{f,\pm}(w):=\inf\{s\ge 0:\ |x^\pm_s(w)|> R_0\},\qquad
T_{R_0}^{b,\pm}(w):=\inf\{s\ge 0:\ |x^\pm_{-s}(w)|> R_0\},
\]
with the convention $\inf\emptyset=\infty$. We then set
\begin{align}\label{eq:L_almost_trapped_sets}
\A_{L}^{f,\pm}&:=\{w\in \Sigma^\pm:\ |x^\pm(w)|\le R_0,\ T_{R_0}^{f,\pm}(w)\ge L\}\\
\A_{L}^{b,\pm}&:=\{w\in \Sigma^\pm:\ |x^\pm(w)|\le R_0,\ T_{R_0}^{b,\pm}(w)\ge L\}\\
\A_{L}^\pm &:= \A_{L}^{f, \pm} \cup \A_{L}^{b,\pm}.
\end{align}
Thus $\A_{L}^{\pm}$ consists of characteristic covectors whose forward/backward bicharacteristic
remains in $\{|x|\le R_0\}$ for at least time $L$. On the complement
$\Sigma^\pm \cap\{|x|\le {R_0}\}\setminus \A_{L}^{\pm}$ one has a uniform exit bound.

 The construction of an escape function will be done using the half-wave decomposition. We have the following relationship between null-bicharacteristics of $p$ and those of $p^\pm$.
\begin{proposition}
\label{prop:bichar_partition}
    Every null-bicharacteristic for the flow generated by $p$ is a null-bicharacteristic for the flow generated by either $p^+$ or $p^-$, and vice versa. Moreover, Assumption~\ref{ass:GCC} also holds for $p^+$ and $p^-$, with perhaps different constants $L$ and $\delta$.
\end{proposition}
\begin{proof}
Let
\[
\Sigma^\pm := \{p^\pm = 0\}, \qquad \Sigma :=\Sigma^+\cup\Sigma^-.
\]
Since \(p=p^+p^-\), every point of \(\Sigma\) lies on exactly one of the two sheets
\(\Sigma^\pm\).
Fix one sign and work on \(\Sigma^\pm\). For any smooth function \(f\),
\[
H_p f = \{p,f\}=\{p^+p^-,f\}=p^- H_{p^+}f + p^+ H_{p^-}f.
\]
Restricting to \(\Sigma^\pm\), where \(p^\pm=0\), gives
\[
(H_p f)|_{\Sigma^\pm} = p^\mp H_{p^\pm}f.
\]
Hence, as vector fields along \(\Sigma^\pm\),
\[
H_p = p^\mp H_{p^\pm}.
\]
Now on \(\Sigma^\pm\),
\[
p^\mp = (\tau-b^\mp)|_{\tau=b^\pm}=b^\pm-b^\mp,
\]
which is nowhere zero since \(b^+>0>b^-\). Thus \(H_p\) and \(H_{p^\pm}\) differ by a
nowhere-vanishing scalar multiple on \(\Sigma^\pm\).
Therefore, if \(\gamma(s)\subset \Sigma^\pm\) is an integral curve of \(H_p\), then
\[
\dot\gamma(s)=c(\gamma(s))\,H_{p^\pm}(\gamma(s)),
\qquad c:=p^\mp|_{\Sigma^\pm}\approx 1.
\]
Define a new parameter \(r\) by
\[
\frac{dr}{ds}=c(\gamma(s)).
\]
Since \(c\neq 0\), this change of parameter is monotone, and for
\(\widetilde\gamma(r):=\gamma(s(r))\), the chain rule gives
\[
\frac{d}{dr}\widetilde\gamma(r)=H_{p^\pm}(\widetilde\gamma(r)).
\]
Hence \(\widetilde\gamma\) is an integral curve of \(H_{p^\pm}\). The converse is obtained by
reversing the argument.
Thus the null-bicharacteristics of \(p\) are exactly the null-bicharacteristics of \(p^\pm\),
up to monotone reparametrization. 

Finally, we note that GCC holds by a change of variables, since $c\approx 1$ globally.
\end{proof}
\begin{remark}
In order to simplify notation, we will continue to use $L$ and $\delta$ as parameters for the flow of $p^{\pm}$ in what follows.
\end{remark}

\begin{proposition}\label{prop:nb-escape}
If there exist $w\in T^*\R^4\setminus o$, $\delta>0$, and $s'>0$ such that
\begin{equation}\label{eq:esc}
    |x^\pm_{\pm s'}(w)|\ge \max\{2R_0,\ |x|+\delta\},
\end{equation}
then $|x^\pm_{\pm s}(w)|\to\infty$ as $s\to\infty$ (in particular, \eqref{eq:esc} holds for all $s\ge s'$). 
\end{proposition}

\begin{proof}
We treat the forward flow $x_s^+(w)$ associated to
\[
p^+(t,x,\tau,\xi)=\tau-b^+(t,x,\xi),
\]
and write $\varphi_s^+(w)=(t_s,x_s,\tau_s,\xi_s)$. The remaining sign choices are handled by the
same argument after reversing the flow parameter.

Since $b^+$ is positively homogeneous of degree $1$ in $\xi$, the quantity $\partial_\xi b^+$ is
homogeneous of degree $0$ in $\xi$. Consequently the spatial velocity
$\dot x_s=-\partial_\xi b^+(t_s,x_s,\xi_s)$ depends only on the direction of $\xi_s$, and by
rescaling the initial covector we may assume $|\xi_0|=1$ without loss of generality (the $x$--projection
of the bicharacteristic is unchanged).

Our aim is to prove a uniform convexity estimate for the function
\[
r(s):=|x_s|^2
\]
once the trajectory is outside a sufficiently large ball. Hamilton's equations are
\begin{equation}\label{eq:ham_pplus}
\left\{
\begin{aligned}
\dot t_s   &= \partial_\tau p^+(t_s,x_s,\tau_s,\xi_s)=1,\\
\dot x_s   &= \partial_\xi p^+(t_s,x_s,\tau_s,\xi_s)=-\partial_\xi b^+(t_s,x_s,\xi_s),\\
\dot \tau_s&= -\partial_t p^+(t_s,x_s,\tau_s,\xi_s)=\partial_t b^+(t_s,x_s,\xi_s),\\
\dot \xi_s &= -\partial_x p^+(t_s,x_s,\tau_s,\xi_s)=\partial_x b^+(t_s,x_s,\xi_s).
\end{aligned}
\right.
\end{equation}
Differentiating $r(s)$ gives
\[
\dot r(s)=2x_s\cdot \dot x_s,\qquad
\ddot r(s)=2|\dot x_s|^2+2x_s\cdot \ddot x_s.
\]

We now invoke the asymptotic flatness hypotheses (uniformly in $t$) in the region $|x|\ge R_0$.
In particular, for $|x|\ge R_0$ and $|\xi|=1$ one has the standard symbol asymptotics
\begin{equation}\label{eq:b_symbol_bounds}
\partial_\xi b^+(t,x,\xi)=\frac{\xi}{|\xi|}+O(|x|^{-1}),\qquad
\partial_{t\xi}b^+(t,x,\xi)=O(|x|^{-2}),\qquad
\partial_{x\xi}b^+(t,x,\xi)=O(|x|^{-2}),
\end{equation}
together with
\begin{equation}\label{eq:b_more_bounds}
\partial_x b^+(t,x,\xi)=O(|x|^{-2}),\qquad
\partial_{\xi\xi}b^+(t,x,\xi)=O(1),
\end{equation}
where all implicit constants are independent of $t$. The first relation in \eqref{eq:b_symbol_bounds}
implies that, after enlarging $R_0$ if necessary, there exists $c_0>0$ such that
\begin{equation}\label{eq:velocity_lb}
|\dot x_s|=\bigl|\partial_\xi b^+(t_s,x_s,\xi_s)\bigr|\ge c_0
\qquad\text{whenever }|x_s|\ge R_0.
\end{equation}

Next we estimate $\ddot x_s$. Differentiating $\dot x_s=-\partial_\xi b^+(t_s,x_s,\xi_s)$ along the flow
and using $\dot t_s=1$ gives
\[
\ddot x_s
=-\partial_{t\xi}b^+(t_s,x_s,\xi_s)
-\partial_{x\xi}b^+(t_s,x_s,\xi_s)\,\dot x_s
-\partial_{\xi\xi}b^+(t_s,x_s,\xi_s)\,\dot \xi_s.
\]
Substituting $\dot \xi_s=\partial_x b^+(t_s,x_s,\xi_s)$ from \eqref{eq:ham_pplus}, and using
\eqref{eq:b_symbol_bounds}--\eqref{eq:b_more_bounds} together with the bound $|\dot x_s|=O(1)$ from
\eqref{eq:velocity_lb}, we obtain
\[
|\ddot x_s|\lesssim |x_s|^{-2}\qquad\text{whenever }|x_s|\ge R_0,
\]
hence
\begin{equation}\label{eq:radial_error}
|x_s\cdot \ddot x_s|\le |x_s|\,|\ddot x_s|\lesssim |x_s|^{-1}\qquad\text{whenever }|x_s|\ge R_0.
\end{equation}

Combining \eqref{eq:velocity_lb} and \eqref{eq:radial_error} yields, for $|x_s|\ge R_0$,
\[
\ddot r(s)=2|\dot x_s|^2+2x_s\cdot \ddot x_s
\ge 2c_0^2 - C|x_s|^{-1}.
\]
Choosing $R_0$ large enough so that $C/R_0\le c_0^2$ gives a uniform constant $c:=c_0^2>0$ with
\begin{equation}\label{eq:accel_bound_correct}
\ddot r(s)\ge c\qquad\text{whenever }|x_s|\ge R_0.
\end{equation}
Crucially, all constants here are independent of the time variable because the asymptotic flatness
bounds are uniform in $t$.

With \eqref{eq:accel_bound_correct} in hand, we now locate a time at which $r$ is already outside
$R_0^2$ and has positive derivative. By hypothesis, $|x_{s'}|\ge \max\{2R,|x|+\delta\}$ and $R\ge R_0$,
so $r(s')\ge (2R_0)^2>R_0^2$ and $r(s')>r(0)$.

If $r(s)>R_0^2$ for every $s\in[0,s']$, then the mean value theorem gives some $s''\in(0,s')$ with
\[
\dot r(s'')=\frac{r(s')-r(0)}{s'}>0,
\]
and in particular $|x_{s''}|>R_0$. Otherwise, define
\[
\alpha:=\sup\{\,s\in[0,s'):\ r(s)\le R_0^2\,\}.
\]
Continuity of $r$ implies $\alpha<s'$ and $r(\alpha)=R_0^2$, while $r(s)>R_0^2$ for all $s\in(\alpha,s']$.
Applying the mean value theorem on $(\alpha,s')$ produces $s''\in(\alpha,s')$ such that
\[
\dot r(s'')=\frac{r(s')-r(\alpha)}{s'-\alpha}=\frac{r(s')-R_0^2}{s'-\alpha}>0,
\]
and again $|x_{s''}|>R_0$.

Thus in all cases we have found $s''\in(0,s')$ with
\begin{equation}\label{eq:critical_correct}
|x_{s''}|>R_0\qquad\text{and}\qquad \dot r(s'')>0.
\end{equation}

We now propagate this information forward. Since $|x_{s''}|>R_0$, the convexity bound
\eqref{eq:accel_bound_correct} applies at $s''$. Moreover, as long as $|x_s|\ge R_0$ we have
$\ddot r(s)\ge c>0$, so $\dot r$ is strictly increasing. In particular, \eqref{eq:critical_correct}
implies $\dot r(s)\ge \dot r(s'')>0$ for all $s\ge s''$ for which $|x_s|\ge R_0$. This forces $r(s)$
to be strictly increasing on such intervals, and therefore $r(s)\ge r(s'')>R_0^2$ for all $s\ge s''$.
Consequently $|x_s|>R_0$ for every $s\ge s''$, so the lower bound \eqref{eq:accel_bound_correct} holds
globally for all $s\ge s''$.

Integrating \eqref{eq:accel_bound_correct} twice for $s\ge s''$ yields
\[
\dot r(s)\ge \dot r(s'')+c(s-s''),\qquad
r(s)\ge r(s'')+\dot r(s'')(s-s'')+\frac{c}{2}(s-s'')^2.
\]
The quadratic term forces $r(s)\to\infty$ as $s\to\infty$, hence $|x_s|\to\infty$. Since $s''<s'$,
the same conclusion holds in particular for all $s\ge s'$. 
\end{proof}

\subsection{$L$-trapped Symbol Construction}\label{sec:semi_trapped}
We begin by constructing a symbol on the $L$-trapped set $\A^\pm_{L}$.
\begin{lemma}\label{lem:semitrap}
Define
\[
\alpha_s^\pm(w):=a(x_s^\pm(w)),
\]
and
\[
q^\pm(w):=\int_0^L (L-s)\big(\alpha_s^\pm(w)-\alpha_{-s}^\pm(w)\big)\,ds.
\]
Then there exist neighborhoods $V^\pm$ of $\A_{L}^{\pm}$ in $\Sigma^\pm$ and constants $C,c>0$, such that
\[
H_{p^\pm}q^\pm + C a \ge c\1_{V^\pm}.
\]
\end{lemma}

\begin{proof}
Let $U^\pm$ be an open neighborhood of $\A_{L}^\pm$ in $\Sigma^\pm$. For each fixed $s\in\R$, set
\[
\alpha_s^\pm(w):=a(x_s^\pm(w))=a(\varphi_s^\pm(w)).
\]
Using the flow property
\[
\varphi_s^\pm(\varphi_h^\pm(w))=\varphi_{s+h}^\pm(w),
\]
we obtain
\[
\alpha_s^\pm(\varphi_h^\pm(w))=\alpha_{s+h}^\pm(w),
\]
hence
\[
H_{p^\pm}\alpha_s^\pm(w)=\partial_s\alpha_s^\pm(w),
\qquad
H_{p^\pm}\alpha_{-s}^\pm(w)=-\partial_s\alpha_{-s}^\pm(w).
\]
Therefore
\[
H_{p^\pm}\big(\alpha_s^\pm-\alpha_{-s}^\pm\big)
=\partial_s\big(\alpha_s^\pm+\alpha_{-s}^\pm\big).
\]

Differentiating under the integral and integrating by parts yields
\begin{align*}
H_{p^\pm}q^\pm(w)
&=\int_0^L (L-s)\,\partial_s\big(\alpha_s^\pm(w)+\alpha_{-s}^\pm(w)\big)\,ds \\
&=\Big[(L-s)\big(\alpha_s^\pm(w)+\alpha_{-s}^\pm(w)\big)\Big]_{s=0}^{s=L}
  +\int_0^L \big(\alpha_s^\pm(w)+\alpha_{-s}^\pm(w)\big)\,ds \\
&=-2L\,\alpha_0^\pm(w)
  +\int_0^L \alpha_s^\pm(w)\,ds
  +\int_0^L \alpha_{-s}^\pm(w)\,ds \\
&=-2L\,a(w)
  +\int_0^L a(x_s^\pm(w))\,ds
  +\int_0^L a(x_{-s}^\pm(w))\,ds.
\end{align*}

Set
\[
U_0^\pm:=U^\pm\cap \Bigl\{w\in\Sigma^\pm:\ a(w)\le \frac{\delta}{4L}\Bigr\}.
\]
Fix $w_0\in \Omega_{R,L}^\pm\cap U_0^\pm$. If $w_0\in \A_{L}^{f,\pm}$, then by definition
\[
x^\pm_{[0,L]}(w_0)\subset \{|x|\le R\},
\]
so GCC gives
\[
\int_0^L a(x_s^\pm(w_0))\,ds \ge \delta.
\]
If instead $w_0\in \A_{L}^{b,\pm}$, then
\[
x^\pm_{[-L,0]}(w_0)\subset \{|x|\le R\},
\]
so GCC gives
\[
\int_0^L a(x_{-s}^\pm(w_0))\,ds \ge \delta.
\]
In either case,
\[
H_{p^\pm}q^\pm(w_0)
\ge \delta - 2L\cdot \frac{\delta}{4L}
= \frac{\delta}{2}.
\]
By continuity, there exists an open neighborhood $U_{w_0}\subset U^\pm$ of $w_0$
such that
\[
H_{p^\pm}q^\pm(w)\ge \frac{\delta}{4}
\qquad \text{for all } w\in U_{w_0}\cap U_0^\pm.
\]
Define
\[
V_0^\pm:=\Bigl(\bigcup_{w_0\in \A_{L}^\pm\cap U_0^\pm}U_{w_0}\Bigr)\cap U_0^\pm.
\]
Then $V_0^\pm$ is an open neighborhood of
\[
\A_{L}^\pm\cap \Bigl\{a\le \frac{\delta}{4L}\Bigr\}
\]
in $U^\pm$, and
\[
H_{p^\pm}q^\pm\ge \frac{\delta}{4}
\qquad \text{on } V_0^\pm.
\]

Now consider the complementary region
\[
V_1^\pm:=U^\pm\cap \Bigl\{w\in\Sigma^\pm:\ a(w)\ge \frac{\delta}{4L}\Bigr\}.
\]
Since \(a\ge 0\), the identity above implies
\[
H_{p^\pm}q^\pm(w)\ge -2L\,a(w).
\]
Hence
\[
H_{p^\pm}q^\pm(w)+C a(w)\ge (C-2L)a(w).
\]
Using \(a(w)\ge \delta/(4L)\) on \(V_1^\pm\), we obtain
\[
H_{p^\pm}q^\pm(w)+C a(w)
\ge (C-2L)\frac{\delta}{4L}.
\]
It suffices to choose \(C=3L\), in which case
\[
H_{p^\pm}q^\pm + C a \ge \frac{\delta}{4}
\qquad\text{on }V_1^\pm.
\]

Finally set
\[
V^\pm:=V_0^\pm\cup V_1^\pm,
\qquad
c:=\frac{\delta}{4}.
\]
Then $V^\pm$ is an open neighborhood of $\A_{L}^\pm$ in $\Sigma^\pm$ and
\[
H_{p^\pm}q^\pm + C a \ge c \1_{V^\pm}
\]
This proves the lemma.
\end{proof}

\subsection{Interior Non-Trapped Set}\label{sec:non_trapped}
With the escape function constructed for interior semi-trapped null-bicharacteristics, we move to constructing an escape function on the remainder of the interior set. The construction here follows that of \cite{kof23} adapted to $L$-trapped sets. To begin we construct a function used to control errors in the exterior region.

\begin{proposition}\label{prop:f}
    Let $\sigma > 0$. Then there exists $f \in C^\infty$, satisfying $f(r) \approx_\sigma 1$ when $r > R_0$ and $f'(r) \approx \sigma c_j 2^{-j} f(r)$ when $r \approx 2^j > R_0$.
\end{proposition}

\begin{lemma}\label{lem:nontrap}
    Let $R_0$ be from \eqref{eq:R_0}. Then there exist $q^\pm \in C^\infty(S^*\R^4)$ and $W^\pm \subset \Sigma^\pm \setminus \A^\pm_{L}$ such that $V^\pm_{0} \cup W^\pm = \Sigma^\pm$ and
    \begin{equation}
        H_{p^\pm} q^\pm \gtrsim c_j 2^{-j} \1_{W^\pm}, \quad \abs{x} \approx 2^j.
    \end{equation}
    In addition, $q^\pm = \e q_{in}^\pm + q^\pm_{\mathrm{ext}}$, where $q^\pm_{in} \in C^\infty(S^*\R^4)$ is supported in $\{\abs{x} \leq 4R_0\}$, $q^\pm_{\mathrm{ext}} \in S^0_{hom}(S^*\R^4)$, and $\e > 0$ is sufficiently small.
\end{lemma}

\begin{proof}
We prove the $+$--case and suppress superscripts. Thus $\varphi_s=\varphi_s^+$, $p=p^+$,
$b=b^+$, and $H_p=H_{p^+}$.

Let $R>R_0$ and let $L>0$ be from the GCC. We use the $L$-trapped set at
radius $2R_0$, which we denote by $\A_{2R_0,L}$. Additionally, we define the set of null-bicharacteristics escaping in uniform time $L$ by $\mathcal{E}_L := \Sigma \setminus \A_{L}$.

Let $V_{2R_0}\subset\Sigma$ be an open neighborhood containing \(\A_{2R_0, L} \cap S^*\R^4\). Define the interior non-trapped region
\[
U_{R_0} := \bigl(\mathcal{E}_L \cap \{|x|\le R_0\}\bigr)\setminus V_{2R_0},
\qquad
W := U_{R_0}\cup\{|x|>R_0\},
\]
so that $V_{2R_0}\cup W=\Sigma^+$.

Choose $\psi\in C^\infty(S^*\R^4)$ with $0\le \psi\le 1$ such that
\[
\supp\psi \subset \Omega_\infty\cap\Sigma^+\cap\{|x|\le 2R_0\},\qquad
\psi\equiv 1 \ \text{on } U_{R_0},\qquad
\supp\psi\cap V_{2R_0}=\emptyset.
\]
With this choice, every $w\in\supp\psi$ is \emph{not} $L$-almost trapped at radius $2R_0$,
hence $T_{2R_0}^{f}(w)<L$. Therefore there exists $s_w\in[0,L)$ such that $|x_{s_w}(w)|> 2R_0$.
Since also $|x(w)|\le R_0$ on $\supp\psi$, Proposition~\ref{prop:nb-escape} (with $\delta:=2R_0$) applies
and yields that $|x_s(w)|\to\infty$ as $s\to\infty$, and in particular $|x_s(w)|> 2R_0$
for all $s\ge s_w$. Because $\supp\psi\subset \{|x|\le 2R_0\}$, it follows that
\begin{equation}\label{eq:psi_terminal_vanish}
\psi(\varphi_L(w))=0\qquad \text{for all } w\in\supp\psi.
\end{equation}

We now define the interior escape function
\[
q_{\mathrm{in}}(w):=-\chi_{<2R}(|x|)\int_0^L \psi(\varphi_s(w))\,ds.
\]
Since the integrand is smooth and the integration is over a finite interval, $q_{\mathrm{in}}$ is a
smooth function on $S^*\R^4$. Moreover, $q_{\mathrm{in}}$ is supported where $\chi_{<2R_0}\neq 0$, hence
in $\{|x|\le 4R_0\}$.

To compute $H_p q_{\mathrm{in}}$, we use that for the forward flow $\varphi_s$,
\[
H_p(\psi\circ\varphi_s)=\frac{d}{ds}(\psi\circ\varphi_s).
\]
Applying $H_p$ to $q_{\mathrm{in}}$ and using the product rule gives
\[
H_p q_{\mathrm{in}}
=-(H_p\chi_{<2R})\int_0^L \psi(\varphi_s)\,ds
-\chi_{<2R}\int_0^L H_p(\psi\circ\varphi_s)\,ds.
\]
The second integral is evaluated by the fundamental theorem of calculus:
\[
\int_0^L H_p(\psi\circ\varphi_s)\,ds
=\int_0^L \frac{d}{ds}\psi(\varphi_s)\,ds
=\psi(\varphi_L)-\psi.
\]
Substituting this yields
\begin{equation}\label{eq:Hp_qin_narrative_buffer}
H_p q_{\mathrm{in}}
=\chi_{<2R_0}\psi
-\chi_{<2R_0}\psi(\varphi_L)
-(H_p\chi_{<2R_0})\int_0^L \psi(\varphi_s)\,ds.
\end{equation}
On $U_{R_0}$ we have $|x|\le R_0$, so $\chi_{<2R_0}\equiv 1$ and $H_p\chi_{<2R_0}\equiv 0$. Moreover
$\psi\equiv 1$ on $U_{R_0}$, and \eqref{eq:psi_terminal_vanish} implies $\psi(\varphi_L)=0$ on $U_{R_0}$.
Thus \eqref{eq:Hp_qin_narrative_buffer} reduces to
\[
H_p q_{\mathrm{in}}=1 \quad \text{on } U_{R_0}.
\]

The remaining terms in \eqref{eq:Hp_qin_narrative_buffer} are supported where
$\nabla\chi_{<2R_0}\neq 0$, i.e. on the fixed overlap annulus $\{2R\lesssim |x|\lesssim 4R_0\}$.
On this annulus we have the crude bound
\[
\Bigl|\int_0^L \psi(\varphi_s)\,ds\Bigr|\le L\|\psi\|_{L^\infty}\le L,
\]
and $|H_p\chi_{<2R_0}|$ is bounded by a constant depending only on $R_0$ (and the uniform bounds on the
coefficients of $H_p$ in $\{|x|\le 4R_0\}$), so the boundary term is uniformly bounded there.

We now introduce the exterior escape function. Let $f$ be as in Proposition~\ref{prop:f} and set
\[
q_{\mathrm{ext}}(w):=-\chi_{>R_0}(|x|)\,f(|x|)\,(\nabla_\xi b)\cdot \frac{x}{|x|}.
\]
A standard Morawetz-type computation, using asymptotic flatness and choosing the parameter in $f$
sufficiently large, yields the lower bound
\[
H_p q_{\mathrm{ext}} \gtrsim c_j 2^{-j}\chi_{>R_0}(|x|)
\qquad \text{on } |x|\approx 2^j,
\]
with constants uniform in $t$.

Finally we glue the interior and exterior pieces. Define
\[
q:=\epsilon q_{\mathrm{in}}+q_{\mathrm{ext}},
\]
with $\epsilon>0$ chosen sufficiently small. On $U_R$ we have $|x|\le R_0$, hence
$\chi_{>R_0}(|x|)=0$ and $H_p q=\epsilon H_p q_{\mathrm{in}}=\epsilon$, giving a uniform
positive lower bound on $U_{R_0}\subset W$. On $\{|x|>4R_0\}$ we have $\chi_{<2R_0}(|x|)=0$, hence
$H_p q=H_p q_{\mathrm{ext}}\gtrsim c_j2^{-j}$ on $|x|\approx 2^j$. In the overlap annulus
$\{2R_0\lesssim |x|\lesssim 4R_0\}$, the boundary term coming from $H_p q_{\mathrm{in}}$ is bounded
in absolute value by $C(R_0)L\,\epsilon$, while $H_p q_{\mathrm{ext}}$ has a fixed positive
lower bound there; choosing $\epsilon$ small ensures that the sum remains positive on this
annulus as well.

Combining these regions, we obtain
\[
H_p q \gtrsim c_j 2^{-j}\,\mathbf{1}_W
\qquad \text{on } |x|\approx 2^j,
\]
where $W=U_{R_0}\cup\{|x|>R_0\}$ and $V_{2R_0}\cup W=\Sigma^+$ by construction. Moreover,
$q_{\mathrm{in}}\in C^\infty(S^*\R^4)$ is supported in $\{|x|\le 4R_0\}$ and
$q_{\mathrm{ext}}\in S^0_{\mathrm{hom}}(S^*\R^4)$, which completes the proof.
\end{proof}

The next step is to combine the constructions in Lemmas \ref{lem:semitrap} and \ref{lem:nontrap} on each
light-cone to produce the desired symbol $q$ needed in Lemma \ref{mainlemma}. Additionally, we will
construct a Lagrangian correction that provides positivity on the elliptic set. These steps, again,
follow the construction in \cite{MST20} and \cite{kof23}, with additional care taken to deal with
the damping term and the non-stationarity.
 
\subsection{Proof of Main Lemma}\label{pf:mainlemma}
We now prove the the main symbol lemma. Namely,
\[H_p q + 2 \gamma \tau a q + pm \geq c \la x \ra^{-2}(\tau^2 + \abs{\xi}^2).\]
\begin{proof}[Proof of Lemma \ref{mainlemma}]
Denote by \(q_1^\pm\) the symbol constructed for the interior trapped set, and by
\(q_2^\pm\) the symbol constructed for the interior non-trapped set. 

Combining on the light-cones, we define
\begin{equation}
    q
    =
    (\tau-b^+)(q_{1}^-+q_{2}^-)
    +
    (\tau-b^-)(q_{1}^++q_{2}^+).
\end{equation}
To simplify notation, write
\begin{equation}
    q_j
    =
    (\tau-b^+)q_{j}^-
    +
    (\tau-b^-)q_{j}^+,
    \qquad j=1,2.
\end{equation}
Set
\[
    e:=H_pq+2\gamma\tau aq.
\]

Then, on the sheet \(\tau=b^\pm\), we have
\[
    q_j|_{\tau=b^\pm}
    =
    (b^\pm-b^\mp)q_{j}^\pm.
\]
Hence
\begin{align}
    e|_{\tau=b^\pm}
    &=
    \sum_{j=1}^2
    \left[
        (H_pq_j)|_{\tau=b^\pm}
        +
        2\gamma b^\pm(b^\pm-b^\mp)a q_{j}^\pm
    \right].
\end{align}

We next compute \((H_pq_j)|_{\tau=b^\pm}\). Since
\[
    p=(\tau-b^+)(\tau-b^-),
\]
on the sheet \(\tau=b^\pm\) we have
\[
    H_p=(b^\pm-b^\mp)H_{p^\pm}.
\]
Therefore
\begin{align}
    (H_p q_j)|_{\tau=b^\pm}
    &=
    (b^\pm-b^\mp)^2 H_{p^\pm}q_{j}^\pm  \nonumber\\
    &\quad
    +(b^\pm-b^\mp)q_{j}^\pm
    H_{p^\pm}(\tau-b^\mp).
\end{align}
Equivalently,
\begin{align}
    (H_p q_j)|_{\tau=b^\pm}
    &=
    (b^+-b^-)^2 H_{p^\pm}q_{j}^\pm  \nonumber\\
    &\quad
    +(b^\pm-b^\mp)q_{j}^\pm
    \left(
        -\partial_t b^\mp+\partial_t b^\pm
        +\nabla_\xi b^\pm\cdot\nabla_x b^\mp
        -\nabla_x b^\pm\cdot\nabla_\xi b^\mp
    \right).
\end{align}
Choosing \(\sigma\) sufficiently large, the derivative falling on
\(e^{\sigma q_j^\pm}\) gives
\begin{align}
    (H_p q_j)|_{\tau=b^\pm}
    \geq
    \frac12 \sigma (b^+-b^-)^2
    q_{j,\lambda}^\pm H_{p^\pm}q_j^\pm
    +
    E_j^\pm,
\end{align}
where \(E_j^\pm\) denotes lower-order error terms which are non-problematic. It is important to note here that, due to \eqref{resc}, the constant $\sigma$ can be chosen independent of $\gamma$.

Now choose \(\gamma \gg\sigma\) sufficiently large. For \(j=1\), Lemma \ref{lem:semitrap} gives
\begin{align}\label{tmp}
    (H_p q_1+2\gamma\tau a q_1)|_{\tau=b^\pm}
    &\gtrsim
    \frac12 \sigma (b^+-b^-)^2q_{1}^\pm
    \left(
        H_{p^\pm}q_1^\pm
        +
        \frac{4\gamma}{\sigma}
        \frac{b^\pm}{b^\pm-b^\mp}a
    \right)  \nonumber\\
    &\gtrsim
    \1_{V_R^\pm}
    (|\xi|^2+\tau^2).
\end{align}
Here
\[
    \frac{b^\pm}{b^\pm-b^\mp}
    =
    \frac{|b^\pm|}{b^+-b^-}
    >0
\]
on the corresponding sheet.

Similarly, for \(j=2\), using Lemma \ref{lem:nontrap} and the positivity of \(a\) on the
relevant region, we get
\begin{align}
    (H_p q_2+2\gamma\tau a q_2)|_{\tau=b^\pm}
    \gtrsim
    \1_{W^\pm}
    c_k2^{-k}(|\xi|^2+\tau^2),
    \qquad |x|\approx 2^k.
\end{align}
Combining the estimates for \(q_1\) and \(q_2\) gives
\begin{equation}
    e|_{\tau=b^\pm}
    =
    (H_pq+2\gamma\tau aq)|_{\tau=b^\pm}
    \gtrsim
    \langle x\rangle^{-2}(|\xi|^2+\tau^2).
\end{equation}

To conclude the proof, we construct a Lagrangian correction which gives positivity
away from the characteristic set. At this stage, it
suffices to work in the spatially compact region
\[
    |x|\leq 2R_0.
\]
In particular, on this region we have \(\langle x\rangle\approx 1\), and the symbol
bounds below are uniform in \((t,x)\).

Set
\(\rho:=\tau^2+|\xi|^2\). From the preceding characteristic-set estimates, we have
\[
    e|_{\Sigma}
    \geq c_\Sigma\rho .
\]

After restricting to \(|x|\leq 2R_0\) and normalizing
\((\tau,\xi)\), the relevant portion of the characteristic set is compact in the phase
variables, with uniform symbol bounds in \(t\). Hence the lower bound above extends
to a fixed conic neighborhood \(U\) of $\Sigma$. Thus,
there exists \(c_0>0\) such that
\[
    e\geq c_0\rho
    \qquad
    \text{on } U.
\]

Choose conic cutoffs \(\chi_+,\chi_-\in S^0\), localizing to neighborhoods of the two
characteristic sheets, such that
\[
    \chi_++\chi_-\equiv 1
\]
on a smaller conic neighborhood \(U_0\Subset U\) of the high-frequency characteristic
set, and such that
\[
    \supp(\chi_++\chi_-)\subset U.
\]
Define
\[
    \chi_{\mathrm{ell}}:=1-\chi_+-\chi_-
\]
 Thus
\(\chi_{\mathrm{ell}}\equiv 0\) on an open neighborhood of \(\Sigma\), and
\(\chi_{\mathrm{ell}}\) is supported in the high-frequency elliptic region. Hence \(p\)
is elliptic on \(\supp\chi_{\mathrm{ell}}\), and
\[
    |p|\gtrsim \tau^2+|\xi|^2
    \qquad
    \text{on } \supp\chi_{\mathrm{ell}}.
\]
We may therefore define
\[
    m:=\chi_{\mathrm{ell}}\frac{C\rho-e}{p},
\]
where \(C>0\) is fixed. Since \(\chi_{\mathrm{ell}}\) vanishes on an open neighborhood
of \(\Sigma\), the quotient is only taken where \(p\neq 0\), and \(m\) extends smoothly
by zero across \(\Sigma\). Moreover, since \(e\in S^2\), \(\rho\in S^2\), and \(p\in S^2\)
is elliptic on \(\supp\chi_{\mathrm{ell}}\), we have
\[
    m\in S^0.
\]

Now compute
\[
    e+pm
    =
    e+\chi_{\mathrm{ell}}(C\rho-e)
    =
    (1-\chi_{\mathrm{ell}})e+C\chi_{\mathrm{ell}}\rho.
\]
Since
\[
    1-\chi_{\mathrm{ell}}=\chi_++\chi_-
\]
is supported in \(U\), where \(e\geq c_0\rho\), we obtain
\[
    e+pm
    \geq
    c_0(1-\chi_{\mathrm{ell}})\rho
    +
    C\chi_{\mathrm{ell}}\rho.
\]
Therefore
\[
    e+pm
    \geq
    \min(c_0,C)\rho.
\]

This proves the desired positivity after adding the Lagrangian correction.
\end{proof}

\section{The two point ILED}

In this section we prove part i) of Theorem~\ref{main}. We start with the following result, which crucially uses the escape function constructed in Section \ref{chap:high_freq}.

\begin{theorem}
\label{thm:high_freq_main}
Let $P$ be an asymptotically flat damped wave operator satisfying Assumption \ref{ass:GCC}, and suppose $\partial_t$ is uniformly timelike and hypersurfaces $\{t = \mathrm{const}\}$ are uniformly spacelike. Then 
\[
\|u\|_{LE^1} \lesssim_{R_0}  \|u\|_{L^2L^2} + \|Pu\|_{LE^*},
\]
for all $u \in \mathcal{S}(\R^{1+3})$  supported in $\{|x|\leq 2R_0\}$.
\end{theorem}

\begin{proof}
We prove the estimate using the escape function constructed in Lemma \ref{mainlemma}.
Let \(q\in S_{hom}^1\) and \(m\in S_{hom}^0\) be the symbols from that lemma. We first cutoff away from zero and define
\[
q_{>1} = \chi_{>1}(|(\tau, \xi)|)q, \quad m_{>1} = \chi_{>1}(|(\tau, \xi)|)q.
\]
We thus have that
\begin{equation}
\label{eq:mainlemma-original-form}
    H_p q_{>1}+2\gamma\tau aq_{>1}+pm_{>1} + r_{-\infty}
    \gtrsim \chi_{>1}(|(\tau, \xi)|)
    \langle x\rangle^{-2}(\tau^2+|\xi|^2).
\end{equation}
where $r_{-\infty}\in S^{-\infty}$.

Let
\[
    Q=q^w-\frac{i}{2}m^w .
\]
Using the commutator identity for \(P=\Box_g+i\gamma aD_t\), we obtain
\begin{equation}
\label{eq:commutator-identity-reduced}
    2\operatorname{Im}\langle Pu,Qu\rangle
    =
    \left\langle
        \operatorname{Op}^w\!\left(H_pq_{>1}+2\gamma\tau aq_{>1}+pm_{>1}\right)u,u
    \right\rangle
    +
    \langle R_1(u), u\rangle
\end{equation}
where $R_1\in OPS^1$; since  $u \in \mathcal{S}(\R^{1+3})$, there are no boundary terms.

We estimate, for all $\delta>0$,
\begin{equation}
\label{eq:remainder-bound}
 |\langle Pu,Qu\rangle| + |\langle R_1(u), u\rangle|
    \leq \delta\|u\|_{LE^1}^2+
   C(\delta) \left(\|u\|_{L^2L^2}^2 + \|Pu\|_{LE^*}^2\right).
\end{equation}

On the other hand, by \eqref{eq:mainlemma-original-form} and Fefferman-Phong, we obtain
\begin{equation}
\label{eq:high-frequency-garding}
    \left\langle
 \operatorname{Op}^w\!\left(H_pq_{>1}+2\gamma\tau aq_{>1}+pm_{>1}\right)u,u
    \right\rangle + 
    C\|u\|_{L^2L^2}^2 
    \gtrsim \|u\|_{LE^1}^2.
\end{equation}

The conclusion follows from \eqref{eq:commutator-identity-reduced}, \eqref{eq:remainder-bound} and \eqref{eq:high-frequency-garding} once we pick a small enough $\delta>0$.
\end{proof}

We will now finish the proof of Theorem~\ref{main} part i). 

Assume first that $Pu\in LE^*$; we will show that
\begin{equation}\label{2ptLE*}
\|u\|_{LE^1[0, T]} + \|\partial u\|_{L^\infty[0, T]L^2}\lesssim  E[u](0)+ E[u](T) + \|Pu\|_{LE^*[0, T]}.
\end{equation}
We pick any $T\ge 4$; the case \(T\le4\) follows from the standard finite-time energy.

Set
\[
    u_{\mathrm{in}}=\chi_{<R_0}u, \quad u_{\mathrm{out}} = \chi_{>R_0/2}u.
\]

It will be enough to obtain the corresponding estimates for $u_{\mathrm{in}}$ and $u_{\mathrm{out}}$.

We first derive the local energy estimate for $u_{\mathrm{out}}$. Using Proposition 3.2 from \cite{MST20} 
we obtain
\begin{equation}
\label{eq:exterior-bound}
    \|u_{\mathrm{out}}\|_{LE^1[0,T]}
    \lesssim\
    \|\partial u(0)\|_{L^2}
    +
    \|\partial u(T)\|_{L^2}
    +
    \|u\|_{L^2[0,T]L^2(A_{\leq R_0})} +
    \|Pu\|_{LE^*[0,T]}
    .
\end{equation}

Next, we derive the local energy estimate for $u_{\mathrm{in}}$. Pick any \(\eta_T\in C_c^\infty((0,T))\) such that
\[
    \eta_T=1 \quad\text{on } [2,T-2],
    \qquad
    \operatorname{supp}\eta_T'\subset [0,2]\cup[T-2,T],
    \qquad
    |\partial_t^k\eta_T|\lesssim_k 1.
\]
We now apply Theorem~\ref{thm:high_freq_main} to
$\eta_T u_{\mathrm{in}}$. We obtain
\[
    \|\eta_Tu_{\mathrm{in}}\|_{LE^1[0,T]}
    \lesssim
    \|\eta_Tu_{\mathrm{in}}\|_{L^2L^2[0,T]}
    +
    \|P(\eta_Tu_{\mathrm{in}})\|_{LE^*[0,T]}.
\]
Expanding,
\[
    P(\eta_Tu_{\mathrm{in}})
    =
    \eta_T Pu_{\mathrm{in}}+[P,\eta_T]u_{\mathrm{in}},
\]
and
\[
    Pu_{\mathrm{in}}
    =
    \chi_{\mathrm{in}}Pu+[P,\chi_{\mathrm{in}}]u.
\]

Therefore
\begin{align}
    \|\eta_Tu_{\mathrm{in}}\|_{LE^1[0,T]}
    \lesssim\;&
    \|u\|_{L^2[0,T]L^2(A_{\leq R_0})}
    +
    \|Pu\|_{LE^*[0,T]}
    +
    \|[P,\chi_{\mathrm{in}}]u\|_{LE^*[0,T]}       \nonumber\\
    &+
    \|[P,\eta_T]u_{\mathrm{in}}\|_{LE^*[0,T]} .
\end{align}

By Hardy's inequality and local-in-time energy estimates, we easily estimate
\begin{equation}\label{terr}
\|[P,\eta_T]u_{\mathrm{in}}\|_{LE^*[0,T]} \lesssim \|u_{\mathrm{in}}\|_{LE^1[0,2]} + \|u_{\mathrm{in}}\|_{LE^1[T-2,T]} \lesssim E[u](0)+E[u](T) + \|Pu\|_{LE^*[0,T]}
\end{equation}

On the other hand,
\begin{equation}\label{outtoin}
\|[P,\chi_{\mathrm{in}}]u\|_{LE^*[0,T]} \lesssim \|\partial u\|_{L^2[0,T]L^2(A_{R_0})} + \|u\|_{L^2[0,T]L^2(A_{R_0})} \lesssim \|u_{\mathrm{out}}\|_{LE^1[0,T]} + \|u\|_{L^2[0,T]L^2(A_{R_0})}
\end{equation}

Due to \eqref{eq:exterior-bound}, \eqref{terr}, and \eqref{outtoin} we obtain
\[
\|\eta_Tu_{\mathrm{in}}\|_{LE^1[0,T]}
    \lesssim E[u](0)+E[u](T) + \|u\|_{L^2[0,T]L^2(A_{\leq R_0})} +\|Pu\|_{LE^*[0,T]}
\]
which by local in time estimates immediately implies
\[
\|u_{\mathrm{in}}\|_{LE^1[0,T]}
    \lesssim E[u](0)+E[u](T) + \|u\|_{L^2[0,T]L^2(A_{\leq R_0})} +\|Pu\|_{LE^*[0,T]}
\]

 To control the supremum of the energy norm, we use the straightforward estimate
 \[
 \|\partial u\|_{L^\infty[0, T]L^2}  \lesssim E[u](0) + \|u\|_{LE^1[0, T]} + \|Pu\|_{LE^*[0,T]}.
 \]

Assume now that $Pu\in L^1L^2$. We now adapt the arguments from the beginning in Section 4 of \cite{MST20}, and construct a parametrix $v$ satisfying $u[0]=v[0]$, and 
\begin{equation}\label{paramL1L2}
E[v](T) + \|v\|_{LE^1[0, T]} + \|P(u-v)\|_{LE^*[0, T]} \lesssim E[u](0) + \|Pu\|_{L^1L^2[0, T]}
\end{equation}

Assuming \eqref{paramL1L2} holds, we can then apply \eqref{2ptLE*} to $u-v$ to finish the proof.

Let $v_{out}^0$ be the solution to
\[
\tilde P v_{out}^0 = Pu, \quad v_{out}^0[0]= u[0]. 
\]
Since local energy estimates hold for the operator $\tilde P$, we have
\[
\|v_{out}^0\|_{LE^1[0, T]} + \|\partial v_{out}^0\|_{L^{\infty}[0,T]L^2}\lesssim E[u](0) + \|Pu\|_{L^1 L^2[0, T]}
\]
Define
\[
v_{out} = \chi_{>R_0}(|x|) v_{out}^0 .
\]
We have
\begin{equation}\label{LEoutL1L2}
\|v_{out}\|_{LE^1[0, T]} + \|\partial v_{out}\|_{L^{\infty}[0,T]L^2}\lesssim E[u](0) + \|Pu\|_{L^1 L^2[0, T]}
\end{equation}

We now construct $v_{in}$ in the interior region $|x|\leq R_0$. WLOG we assume that $T$ is an integer, and write $[0,T]=\cup_{k=0}^{T-1} [k,k+1]$. Let $v_k$ solve $$Pv_k = (1-\chi_{>R_0}(|x|))Pu$$ so that
\[
v_0[0] = (1-\chi_{>R_0}(|x|)) u[0]
\]
\[
v_k[k] = 0, \quad 1\leq k\leq T-1
\]

Let $\chi_k$ be a partition of unity subordinate to $[k, k+2]$ so that $\sum_k \chi_k(t)=1$, and define
\[
v_{in} = \sum_{k=0}^{T-1} \chi_k(t) v_k
\]
By finite speed of propagation and local-in-time energy estimates we have
\begin{equation}\label{LEinL1L2}
\|v_{in}\|_{LE^1[0, T]} + \|\partial v_{in}\|_{L^{\infty}[0,T]L^2}\lesssim E[u](0)+E[u](T) + \|Pu\|_{L^1 L^2[0, T]}
\end{equation}
Finally, we let
\[
v=v_{in} + v_{out}.
\]
By construction, $v[0]=u[0]$. Moreover, \eqref{LEoutL1L2} and \eqref{LEinL1L2} immediately give
\[
\|v\|_{LE^1[0, T]} + \|\partial v\|_{L^{\infty}[0,T]L^2}\lesssim E[u](0) + \|Pu\|_{L^1 L^2[0, T]}
\]
On the other hand, we see that
\[
P(u-v) = [P, \chi_{>R_0}(|x|)]v_{out}^0 +  \sum_{k=0}^{T-1} [P,\chi_k(t)] v_k
\]
We estimate
\[
\|[P, \chi_{>R_0}(|x|)]v_{out}^0\|_{LE^*[0, T]} \lesssim \|v_{out}^0\|_{LE^1[0, T]} \lesssim E[u](0) + \|Pu\|_{L^1 L^2[0, T]}
\]

For the interior term we have when $1\leq k\leq T-1$
\[
\|[P,\chi_k(t)] v_k\|_{LE^*[0, T]} \lesssim \|v_k\|_{LE^1[k, k+2]} \lesssim \|P v_k\|_{LE^*[k, k+2]} \lesssim \| Pu\|_{LE^*[k, k+2]}.
\]
Moreover
\[
\|[P,\chi_0(t)] v_0\|_{LE^*[0, T]} \lesssim E[u](0) + \|v_0\|_{LE^1[0, 2]} \lesssim E[u](0) + \| Pu\|_{L^1L^2[0, 2]}
\]
and after summation
\[
\|\sum_{k=0}^{T-1} [P,\chi_k(t)] v_k\|_{LE^*[0, T]} \lesssim E[u](0) + \|Pu\|_{L^1L^2[0, T]}
\]
This finishes the proof of \eqref{paramL1L2}.

\begin{remark}\label{HFSchw}
Using the arguments above, it is clear that one can now obtain Theorem~\ref{thm:high_freq_main} without the assumption that $u$ is supported in $\{|x|\leq 2R_0\}$.
\end{remark}

\section{Local Energy Decay}
\label{chap:main_theorem}

In this section we prove the second part of our main theorem. Since this is very similar to the previous work \cite{MST20}, we will only sketch the argument.

We first state two results that will be used to control the medium and low frequencies. For the proofs, we refer to \cite{MST20}, Theorems 5.4 and 6.1

\begin{theorem}\label{thm:medium_freq_main}
Let $P$ be an asymptotically flat damped wave operator, and suppose that $\partial_t$ is uniformly timelike. Then, for any $\delta > 0$, there exists a bounded, non-decreasing radial weight $\varphi = \varphi(\ln(1 + r))$ such that for all $u \in \mathcal{S}(\mathbb{R}^4)$, we have
\begin{multline}\label{eq:medium_freq_estimate}
\left\|(1 + \varphi'')^{1/2} e^\varphi (\nabla_x u, \langle r \rangle^{-1}(1 + \varphi') u)\right\|_{LE} + \left\|(1 + \varphi')^{1/2} e^\varphi \partial_t u\right\|_{LE} \\
\lesssim \|e^\varphi P u\|_{LE^*} + \delta \left\|(1 + \varphi')^{1/2} e^\varphi u\right\|_{LE} + \left\|\langle r \rangle^{-1} (1 + \varphi'')^{1/2}(1 + \varphi') e^\varphi \partial_t u\right\|_{LE}.
\end{multline}
\end{theorem}

The parameter $\delta > 0$ allows us to absorb error terms arising from the medium-frequency cutoff, and its presence is essential for controlling arbitrarily large time-frequency ranges while keeping the weight $\varphi$ bounded.

The main theorem in the low-frequency regime is as follows:
\begin{theorem}
\label{thm:low_freq_main}
    Let $P$ be an asymptotically flat damped wave operator, and suppose that $\pd_t$ is uniformly
    timelike. Then,
    \begin{equation}
        \norm{u}_{LE^1} \lesssim \norm{\pd_t u}_{LE^1_c} + \norm{Pu}_{LE^*}
    \end{equation}
    for all $u \in \mathcal{S}(\R^4)$. 
\end{theorem}

The main result of the section is the following:
\begin{theorem}[Local energy decay]
\label{thm:main_led}
Let $P$ be an asymptotically flat damped wave operator satisfying the uniform geometric control condition. Assume $\partial_t$ is uniformly timelike and the constant-time hypersurfaces are uniformly spacelike. Assume in addition that the metric is slowly varying in time in the sense of Assumption~\ref{asmp:sv}. Then
\begin{equation}\label{eq:main_led}
  \|u\|_{LE^{1}}\lesssim \|Pu\|_{LE^{*}}
\end{equation}
for all $u \in \mathcal{S}(\R^{1+3})$.
\end{theorem}

\begin{proof}

The proof proceeds by decomposing solutions into low, medium, and high time-frequency components, applying the previous estimates  to each piece, and then controlling the resulting commutator errors.

Fix thresholds $0<\tau_l\ll 1\ll \tau_h$, and choose an even function $\eta\in C_c^\infty(\R)$ with
\[
0\le \eta\le 1,\qquad
\eta(s)=1\ \text{for }|s|\le \tfrac12,\qquad
\eta(s)=0\ \text{for }|s|\ge 1.
\]
The time-frequency symbols are then defined by
\[
q_l(\tau):=\eta(\tau/\tau_l),\qquad
q_{\le h}(\tau):=\eta(\tau/\tau_h),
\]
\[
q_h(\tau):=1-q_{\le h}(\tau),\qquad
q_m(\tau):=q_{\le h}(\tau)-q_l(\tau).
\]
The symbol $q_l$ localizes to $|\tau|\lesssim \tau_l$, $q_m$ localizes to the intermediate band $\tau_l\lesssim |\tau|\lesssim \tau_h$, and
\[
q_h(\tau)=1 \quad \text{for } |\tau|\ge \tau_h
\qquad\text{and}\qquad
q_h(\tau)=0 \quad \text{for } |\tau|\le \tfrac12\tau_h.
\]
For clean separation one may assume $\tau_l\le \tau_h/4$ so that $q_m\ge 0$.

Quantizing via the Weyl calculus gives operators
\[
Q_l:=\opw(q_l),\qquad Q_m:=\opw(q_m),\qquad Q_h:=\opw(q_h),
\]
satisfying $I=Q_l+Q_m+Q_h$. For any $u$, we therefore have the decomposition
\[
u = Q_l u + Q_m u + Q_h u =: u_l+u_m+u_h.
\]

By Remark~\ref{HFSchw} and Plancherel's formula, we have
\[
\|u_h\|_{LE^1}\lesssim \|u_h\|_{L^2L^2(|x|\leq 2R_0)} + \|Pu_h\|_{LE^*} \lesssim \tau_h^{-1}\|u_h\|_{LE^1} + \|Pu_h\|_{LE^*}
\]
and by picking $\tau_h$ large enough, we obtain
\[
\|u_h\|_{LE^1}\lesssim \|Pu_h\|_{LE^*}
\]

Similarly, Theorem~\ref{thm:low_freq_main} and Plancherel's formula yield
\[
\|u_l\|_{LE^1} \lesssim \|\partial_t u_l\|_{LE^1_c} + \|Pu_l\|_{LE^*} \lesssim \tau_l\|u_l\|_{LE^1} + \|Pu_l\|_{LE^*}
\]
and by picking $\tau_l$ small enough, we obtain
\[
\|u_l\|_{LE^1}\lesssim \|Pu_l\|_{LE^*}
\]
Finally, choosing $\delta>0$ small enough (depending on $\tau_h$ and $\tau_l$) in Theorem~\ref{thm:medium_freq_main}, we obtain
\[
\|u_m\|_{LE^1} \lesssim \|Pu_m\|_{LE^*}
\]

We now use Proposition 7.2 in \cite{MST20} to control the commutators, which states that
\[
\|[P, \eta^w]u\|_{LE^*} \lesssim \epsilon\left(\|u\|_{LE^1} + \|Pu\|_{LE^*}\right).
\]
We note here that since the damping term $a(x)$ is time-independent, the damping term commutes with $\eta^w$.

After rescaling, one easily obtains for all $i\in\{l, m, h\}$ that
\[
\|[P, Q_i]\|_{LE^*} \lesssim \epsilon\left(\|u\|_{LE^1} + \|Pu\|_{LE^*}\right)
\]
For $\epsilon$ small enough the commutators can thus be absorbed on the left hand side. This finishes the proof.

\end{proof}

\subsection{Local Energy Decay with Initial Data}

We now establish the full local energy decay estimate, which incorporates initial data. Once again, we will use the ideas from \cite{MST}.

It is enough prove the two-point estimate 
\begin{equation}
\|u\|_{LE^1[0, T]}  \lesssim  E[u](0)+ E[u](T) + \|Pu\|_{LE^*[0, T] + L^1 L^2[0, T]}.
\end{equation}

Indeed, due to Assumption~\ref{asmp:sv} and energy estimates, we obtain
\[
E[u](T)\lesssim \|\partial u\|_{L^\infty[0, T]L^2}  \lesssim E[u](0) +  C(\epsilon) \|Pu\|_{LE^*[0, T] + L^1 L^2[0, T]} + \epsilon \left(\|u\|_{LE^1[0, T]} + \|\partial u\|_{L^\infty[0, T]L^2} \right) 
\]
and the last term can be absorbed for small enough $\epsilon$.

Assume first that $Pu\in LE^*$; we will show that
\begin{equation}\label{2pt}
\|u\|_{LE^1[0, T]}  \lesssim  E[u](0)+ E[u](T) + \|Pu\|_{LE^*[0, T]}.
\end{equation}
We will construct a parametrix $v$ satisfying $u[0]=v[0]$, $u[T]=v[T]$ and 
\begin{equation}\label{param}
\|v\|_{LE^1[0, T]} + \|P(u-v)\|_{LE^*[0, T]} \lesssim E[u](0) + \|Pu\|_{LE^*[0, T]}
\end{equation}

The conclusion \eqref{2pt} now follows after applying Theorem~\ref{thm:main_led} to $u-v$.

We first construct an approximate solution in the exterior region $|x|\geq R_0$.

Let $v_{out}^0$ be the solution to
\[
\tilde P v_{out}^0 = Pu, \quad v_{out}^0[0]= u[0]. 
\]
Since local energy estimates hold for the operator $\tilde P$, we have
\[
\|v_{out}^0\|_{LE^1[0, T]} + \|\partial v_{out}^0\|_{L^{\infty}[0,T]L^2}\lesssim E[u](0) + \|Pu\|_{LE^*[0, T]}
\]
Similarly, let $v_{out}^T$ to be the solution to
\[
\tilde P v_{out}^T = Pu, \quad v_{out}^0[T]= u[T]. 
\]
which satisfies
\[
\|v_{out}^T\|_{LE^1[0, T]} + \|\partial v_{out}^T\|_{L^{\infty}[0,T]L^2}\lesssim E[u](T) + \|Pu\|_{LE^*[0, T] + L^1 L^2[0, T]}
\]

We now let $\chi_T(t)$ be identically $1$ for $0\leq t\leq T/2$ and supported in $t\leq 3T/4$. Define
\[
v_{out} = \chi_{>R_0}(|x|) \left(\chi_{T}(t)v_{out}^0 + (1-\chi_{T}(t))v_{out}^T\right).
\]
Clearly 
\begin{equation}\label{LEout}
\|v_{out}\|_{LE^1[0, T]} + \|\partial v_{out}\|_{L^{\infty}[0,T]L^2}\lesssim E[u](0) + E[u](T) + \|Pu\|_{LE^*[0, T]}
\end{equation}

We now construct $v_{in}$ to be an approximate solution in the interior region $|x|\leq R_0$. WLOG we assume that $T$ is an integer, and write $[0,T]=\cup_{k=0}^{T-1} [k,k+1]$. Let $v_k$ solve $$Pv_k = (1-\chi_{>R_0}(|x|))Pu$$ so that
\[
v_0[0] = (1-\chi_{>R_0}(|x|)) u[0]
\]
\[
v_{T-1}[T] = (1-\chi_{>R_0}(|x|)) u[T]
\]
\[
v_k[k] = 0, \quad 1\leq k\leq T-1
\]

Let $\chi_k$ be a partition of unity subordinate to $[k, k+2]$ so that $\sum_k \chi_k(t)=1$, and define
\[
v_{in} = \sum_{k=0}^{T-1} \chi_k(t) v_k
\]
By finite speed of propagation and local-in-time energy estimates we have
\begin{equation}\label{LEin}
\|v_{in}\|_{LE^1[0, T]} + \|\partial v_{in}\|_{L^{\infty}[0,T]L^2}\lesssim E[u](0)+E[u](T) + \|Pu\|_{LE^*[0, T]}
\end{equation}
Finally, we let
\[
v=v_{in} + v_{out}.
\]
By construction, $v[0]=u[0]$ and $v[T]=u[T]$. Moreover, \eqref{LEout} and \eqref{LEin} immediately give
\[
\|v\|_{LE^1[0, T]} + \|\partial v\|_{L^{\infty}[0,T]L^2}\lesssim E[u](0)+E[u](T) + \|Pu\|_{LE^*[0, T]}
\]

Finally, we see that
\[
P(u-v) = [P, \chi_{>R_0}(|x|) \chi_{T}(t)]v_{out}^0 + [P, \chi_{>R_0}(|x|) (1-\chi_{T}(t))]v_{out}^T + \sum_{k=0}^{T-1} [P,\chi_k(t)] v_k
\]
and we estimate
\[
\|[P, \chi_{>R_0}(|x|) \chi_{T}(t)]v_{out}^0\|_{LE^*[0, T]} \lesssim \|v_{out}^0\|_{LE^1[0, T]} \lesssim E[u](0) + \|Pu\|_{LE^*[0, T]}
\]
\[
\|[P, \chi_{>R_0}(|x|) (1-\chi_{T}(t))]v_{out}^T\|_{LE^*[0, T]} \lesssim \|v_{out}^T\|_{LE^1[0, T]} \lesssim E[u](T) + \|Pu\|_{LE^*[0, T]}
\]
For the interior term we have when $1\leq k\leq T-2$
\[
\|[P,\chi_k(t)] v_k\|_{LE^*[0, T]} \lesssim \|v_k\|_{LE^1[k, k+2]} \lesssim \|P v_k\|_{LE^*[k, k+2]} \lesssim \| Pu\|_{LE^*[k, k+2]}.
\]
Moreover
\[
\|[P,\chi_0(t)] v_0\|_{LE^*[0, T]} \lesssim E[u](0) + \|v_0\|_{LE^1[0, 2]} \lesssim E[u](0) + \| Pu\|_{LE^*[0, 2]}
\]
\[
\|[P,\chi_{T-1}(t)] v_{T-1}\|_{LE^*[0, T]} \lesssim E[u](T) + \|v_{T-1}\|_{LE^1[T-2, T]} \lesssim E[u](T) + \| Pu\|_{LE^*[T-2, T]}
\]
and after summation
\[
\|\sum_{k=0}^{T-1} [P,\chi_k(t)] v_k\|_{LE^*[0, T]} \lesssim E[u](0)+E[u](T) + \|Pu\|_{LE^*[0, T]}
\]
This finishes the proof of \eqref{2pt}. 

If $Pu\in L^1L^2$ we use the parametrix from the previous section and \eqref{2pt} to obtain the desired estimate.

\bibliographystyle{plain}
\bibliography{dissertation}
\end{document}